\documentclass[12pt]{article}

\usepackage{amsmath, amsfonts, amssymb, amsthm,  enumerate, setspace}

\theoremstyle{plain}
\newtheorem{theorem}{Theorem}
\newtheorem{lemma}[theorem]{Lemma}
\newtheorem{corollary}[theorem]{Corollary}
\newtheorem{proposition}[theorem]{Proposition}

\theoremstyle{definition}

\theoremstyle{remark}
\newtheorem{remark}[theorem]{Remark}

\def\Z{\ensuremath{\mathbb {Z}}}

\def\({\left(}
\def\){\right)}

\def\al{\alpha}
\def\abs#1{\vert\, #1\,\vert}

\def\be{\beta}

\def\iy{\infty}

\def\const{{\operatorname{const}}}
\def\modd{{\operatorname{mod}}}

\def\mC2#1{{\mathbf C}^{(2)}(#1)}
\def\mC#1#2{{\mathbf C}^{(#1)}(#2)}

\def\cM{{\mathfrak M}}

\def\fP{{\mathfrak p}}

\def\cH{{\cal H}}

\def\cP{{\mathcal {P}}}
\def\cQ{{\mathcal {Q}}}

\def\bydef{:=}

\def\vf{\varphi}

\def\d{\,\mathrm{d}}

\def\text#1{\mbox{ #1 }}

\def\i0t{\int_0^t}

\def\beq{\begin{equation}}
\def\eeq{\end{equation}}
 
\title{\bf Combinatorial Properties of Mills' Ratio.}
\author{Alexander Kreinin\\
\small Quantitative Research, Risk Analytics, IBM \\  [-0.8ex]
\small 185 Spadina Ave., Toronto, Ontario, M5T2C6, Canada \\  [-0.8ex]
\small\tt alex.kreinin@ca.ibm.com}
\date{Oct 1, 2012}
 
\begin{document}
\maketitle
\begin{abstract}
We consider combinatorial
properties of the Mills' ratio, $R(t)=\int_t^\iy\vf(x)\d x/\vf(t)$, 
where $\vf(t)$ is the standard normal density. We explore the interplay between
 a continued fraction expansion for the Mills' ratio, the Laplace polynomials 
and a new family of combinatorial identities.  

\end{abstract}
 
% \maketitle

%%% \renewcommand\theequation{\thesection.\arabic{equation}}  %#eq-s: sect.num
%%%             \catcode`@=11 \@addtoreset{equation}{section} \catcode`@=12

\section{Introduction}
Combinatorial identities often have striking
relations to the special functions. 
Let us consider the following two seemingly unrelated identities,
\beq
    \sum_{n=0}^\iy \sum_{m=0}^\iy \sum_{j=0}^{n}
    \frac{\(n+m\)!}{m! j! (m+2n+1-j)!}\,\, 2^{-n}  %% {{m+2n+1}\choose j}
    =\sqrt{2\pi} e^2 \(\Phi(2)-\Phi(1)\) ,  \label{eq_ident1}
\eeq
and
\beq
\sum_{k=0}^n (-1)^k \frac{1}{2k+1} {n\choose k}=\frac{2^{2n} \(n!\)^2 }{(2 n+1)! },
\label{eq_ident2}
\eeq
where $\Phi(\cdot)$ is the standard Normal distribution function\footnote{The first identity is new, 
to the best of our knowledge.
The second identity is known (see \cite{Riordan_CI}).}.
In this paper we show that these identities actually have common combinatorial nature linked to
the continued fraction for the Mills' ratio, %%  \cite{Mills},
$$ R(t) = \frac{\bar\Phi(t)}{\vf(t)}, $$
 where
$\vf(t)$ is the standard Normal density and $\bar\Phi(t)=\int_t^\iy\vf(s)\d s$
is the tail of the Normal distribution. 
%% The function
%% $$ h(t)=\frac{1}{R(t)}=\frac{\vf(t)}{\bar\Phi(t)}$$
%% is usually called the hazard rate of the Normal distribution.
The function $R(t)$  plays an important role in Probability Theory, Statistics,
Stochastic Analysis and many applied areas including Queueing theory, Reliability and
Mathematical Finance\footnote{Examples can be found in 
\cite{Feller70}, \cite{McKean}, \cite{Abram}, \cite{KendStuart}}.

Despite that the function $R(t)$ was named after John Mills, who tabulated its values \cite{Mills},
the first known statements about  $R(t)$ appeared, in fact, more than $200$ years
ago by Laplace  \cite{Laplace} who found the asymptotic expansion
\beq
  R(t)\sim \frac1t -\frac1{t^3} + \frac{1\cdot 3}{t^5} -\frac{1\cdot 3\cdot 5}{t^7} +\dots
\quad\text{for $t>0$}.
\label{eq_R}
\eeq
and  the continued fraction expansion,
\beq
    R(t) = \cfrac{1}{t+\cfrac{1}{t+\cfrac{2}{t+\cfrac{3}{t+\cfrac{4}{\ddots}}}}} \,.
\label{eq_Cont_frac}
\eeq
Laplace also found the rational approximations, 
$R(t)\sim Q_{k-1}(t)/P_k(t)$ for the function $R(t)$,
where  $P_k(t)$ and $Q_k(t)$ are polynomials of degree $k$  and
 wrote down these polynomials for $k\le 4$.  %%  in what follows
Thirty years later, Jacobi \cite{Jacobi} gave a very short (and rigorous) derivation of the expansion (\ref{eq_R}).
Jacobi  also found recurrent equations for the  polynomials  in the rational approximation.
%% Apparently, Laplace's interest was

Since then, numerous papers and monographs discussing different properties of the Mill's ratio were published.
 A modified proof of asymptotic (\ref{eq_R}) is discussed in \cite{Feller70}. 
More accurate asymptotic expansions for the Mill's ration were obtained in
\cite{Ruben} and in \cite{Shenton}. The latter paper contains also a good account of the 
properties of the tail of a normal distribution.

Recently, a few interesting papers on continued fractions and general approximation schemes for the
Mills ratio were published. We mention here the papers
 \cite{Pinelis}, \cite{Baricz}, \cite{Kouba} and \cite{Dumbgen} analyzing 
 monotonicity properties of some functions involving $R(t)$. Based on these properties 
a series of irrational approximations for $R(t)$ were derived in \cite{Kouba} and \cite{Dumbgen}.
%\cite{Kouba} also found formulae for the coefficients of the polynomials $P_k$ and $Q_k$ and explored the
%relation of the series $P_k$ to the Hermite polynomials. 

In this paper, we discuss combinatorial identities connected to the continuous fractions for the Mills ratio.
In Section~\ref{sec_bounds} we introduce Laplace polynomials whose ratio form the continuous fractions for
the Mills ratio  and study their properties. 
To make this paper self-contained, we re-derive in Section~\ref{sec_bounds} the rational approximations for the function
$R(t)$.

%% Our objective is to obtain as much as possible information about the 

Despite that the recurrent relations for the polynomials $P_k$ and $Q_k$ have been known for
a very long time, their coefficients were not studied until recently.
The first analysis of these coefficients, published in \cite{Kouba},
 appeared only in $2006$, to the best of our knowledge.

It turned out that the properties of the polynomials  $P_k$ closely resemble those of
the Hermite polynomials (see also
\cite{Andrews}). The combinatorial structure  of the coefficients of the polynomials $Q_k$ is more complex.
We find rather  simple formula than the one in \cite{Kouba} in Section~\ref{sec_PM}.

In Section~\ref{sec_gf} the double generating functions of these polynomials is computed. 
From the double generating function we  find a
series of  combinatorial identities that  includes (\ref{eq_ident1}) and (\ref{eq_ident2}). 
In Section~\ref{sec_LH} we discuss Laplace polynomials connections to the Hermite polynomials.
The paper is closed with the relation between the Laplace polynomials and matching numbers in
Section~\ref{sec_LPmn}.

\subsection*{Acknowledgements}I am very  grateful to Ian Iscoe, Sebastian Jaimungal,  Alexey Kuznetsov, Tom Salisbury,
Eugene Seneta,
Isaac Sonin, Hans Tuenter and Vladimir Vinogradov for interesting comments and stimulating discussions.

\section{Polynomials $P_k(t), Q_k(t)$ and inequalities for $R(t)$}\label{sec_bounds}
In this Section, we derive a series of inequalities for the Mills ratio expressed in terms of the
Laplace polynomials, introduced below. %% and introduce

Our approach is based on the following simple idea. The function $R(t)$ can be
represented as a Laplace transform
 of a non-negative function on the positive semi-axis. %% $x\in \R_+$.
According to  the classical theorem\footnote{In fact, we use only a trivial part of this theorem: 
if a function is represented as a Laplace transform of a non-negative function of positive semi-axis, 
then it is completely monotone.} proved by S.~N. Bernstein~\cite{Bernst28},
 the Laplace transform, in this case, is a completely monotone function and, therefore,
satisfies an infinite sequence of alternating inequalities for the derivatives of $R(t)$.
Suppose that $R(t)$ satisfies the differential equation
$$ R^\prime(t)=\al(t) R(t) + \be(t).$$
Then using the complete monotonicity property and relations derived 
from the differential equation one can find an infinite sequence of inequalities for
the function $R(t)$ expressed through the functions $\al(t)$, $\be(t)$ and their 
derivatives.   

Using this idea we obtain in Section~\ref{sec_bounds} a sequence of ``self-improving" inequalities
$$ \frac{Q_{k-2}(t)}{P_{k-1}(t)} \le R(t)\le  
 \frac{Q_{k-1}(t)}{P_k(t)}, \qquad k=2, 4, 6, \dots   
$$
and show that $Q_{k-1}(t)/P_k(t)$ 
represents the rational approximation of the continued fraction (\ref{eq_Cont_frac}).

\begin{lemma}\label{lem1}
 $R(t)$ is a completely monotone function.  \end{lemma}
\begin{proof} 
We have for $t\ge 0$
$$\int_0^\iy e^{-tx} e^{-x^2/2}\d x = \sqrt{2\pi} e^{t^2/2}\int_0^\iy \frac{e^{-(x+t)^2/2}}{\sqrt{2\pi}}\d x
=\frac{\bar\Phi(t)}{\vf(t)}=R(t). $$
The statement of the lemma now follows from the Bernstein's theorem (see \cite{Bernst28}).
\end{proof}
 
%% $$ \hat f(t)=\int_0^\iy e^{-tx} f(x)\d x, \quad t\ge 0. $$
It follows from Lemma~\ref{lem1}
 that $R(t)$ is  
infinitely differentiable on the half-line, $[0, \iy)$, and satisfies the inequality
\beq
 (-1)^n \frac{\d^n  R(t)}{\d t^n}\ge 0, \quad n=1, 2, \dots, t\ge 0. \label{eq_CM_main}
\eeq 
%L 
It is not difficult to verify that the function $R(t)$ satisfies the differential equation
\beq
    \frac{\d  R(t)}{\d t}= t\cdot R(t) -1. 
\label{eq_DE_hf}
\eeq
Since the function $R(t)$ is completely monotone we obtain from (\ref{eq_CM_main}) and (\ref{eq_DE_hf})
for $t>0$
$$ R(t)\le \frac1t. $$
The latter inequality is equivalent to
\beq 
    \bar\Phi(t)\le \frac{\vf(t)}{t}, \quad t>0. \label{eq_fDR}
\eeq
Differentiating $R(t)$ twice and using (\ref{eq_DE_hf}) we obtain
for $t>0$ 
$$
    \frac{\d^2 R(t)}{\d t^2}= R(t) \(t^2+ 1\) - t. 
$$
Then from (\ref{eq_CM_main}), $n=2$ we derive 
$$ \bar\Phi(t)> \frac{\vf(t)}{t+t^{-1}}, $$
and together with (\ref{eq_fDR}) delivers the well-known asymptotic relation
$$ \bar\Phi(t)\sim \vf(t)\cdot t^{-1}\quad\text{as $t\to +\iy$}.$$

Let us now consider the derivatives of the Mill's ratio. From (\ref{eq_DE_hf}) we find
\beq
\frac{\d^k R(t)}{\d t^k}= t\cdot \frac{\d^{k-1} R(t)}{\d t^{k-1}} +
(k-1)\cdot \frac{\d^{k-2} R(t)}{\d t^{k-2}}.
\quad k=1, 2, \dots.
\label{eq_mrec_drvR} 
\eeq
It follows from (\ref{eq_DE_hf}) that the latter equation can be written 
\beq
\frac{\d^k R(t)}{\d t^k}= R(t)\cdot P_k(t) 
   -  Q_{k-1}(t), \qquad k=1, 2, \dots,
\label{eq_fEq_LaplPoly} \eeq
where $P_k(t)$ and $Q_k(t)$ are polynomials of degree $k$.  We shall call 
$P_k(t)$ and $Q_k(t)$ the {\it Laplace polynomials} in
what follows.
The Laplace  polynomials satisfy the following recurrent equations
\begin{eqnarray}
P_{k+1}(t) &=& t P_k(t) + P_k^\prime(t) \label{eq_P_k}\\
Q_k(t) &=& P_k(t) + Q_{k-1}^\prime(t),   \label{eq_Q_k}
\end{eqnarray}
where $P_0(t)=Q_0(t)=1$.
Using  Equations~(\ref{eq_P_k}) and (\ref{eq_Q_k}) one can find 
 $P_k(t)$ and $Q_k(t)$ for any integer $k$ (see Table~\ref{tab_poly}). 
\begin{table}[hbp]
\begin{center}
\begin{tabular}{||c||c||c||}
\hline
  $k$  & $P_k(t)$   & $Q_{k-1}(t)$ \\ \hline
   $1$  & $t$  & $1$ \\ \hline
   $2$ & $t^2+1$ & $t$ \\ \hline
   $3$ & $t^3+3t$  & $t^2+2$ \\ \hline
   $4$ & $t^4+6t^2+3$  & $t^3+5t$ \\ \hline
   $5$ & $t^5+10t^3 +15t$  & $t^4+9 t^2 +8$ \\ \hline
   $6$ & $t^6+15t^4 + 45 t^2 + 15$ & $t^5 + 14 t^3 + 33 t $  \\ \hline
   $7$ & $t^7+21t^5 + 105 t^3 + 105t$  & $t^6 + 20 t^4 + 87 t^2 +48  $ \\ \hline
   $8$ & $t^8 + 28t^6 + 210 t^4 + 420 t^2 + 105$ & $t^7 + 27t^5+185t^3 +279t  $  \\ 
\hline\hline
\end{tabular}
\end{center}
\begin{center}
\begin{minipage}{9cm }
\caption{Laplace polynomials $P_k(t)$ and $Q_{k-1}(t)$.}\label{tab_poly}
\end{minipage}\end{center}
\end{table}
\begin{lemma}\label{lem2}
The Mill's ratio satisfies the inequalities
\beq
\frac{Q_{k-2}(t)}{P_{k-1}(t)} \le R(t)\le  
 \frac{Q_{k-1}(t)}{P_k(t)}, \qquad k=2, 4, 6, \dots   \label{eq_ineq_LapPoly}
\eeq
\end{lemma}
\begin{proof} Inequalities (\ref{eq_ineq_LapPoly}) follow from the complete monotonicity of the function $R(t)$,
and Equations~(\ref{eq_CM_main}) and (\ref{eq_fEq_LaplPoly}).  \end{proof}

\noindent
The following statement on the Laplace polynomials is known for a very long time.
\begin{lemma}[Jacobi, Pinelis, Kouba]\label{lem_pk}
The polynomials $P_k(t)$ and $Q_k(t)$ satisfy the relation
\begin{eqnarray*}
P_{k+1}(t) &=& t P_k(t) + k P_{k-1}(t), \\
Q_{k+1}(t) &=& t Q_k(t) + (k+1) Q_{k-1}(t)
\end{eqnarray*}
\end{lemma}
\begin{proof}
This Lemma is proved by induction.  \end{proof}

\noindent
Lemma~\ref{lem_pk} immediately implies the continued fraction representation (\ref{eq_Cont_frac}). Indeed,
$$ Q_k(t) = t Q_{k-1}(t) + k Q_{k-2}(t). $$ Then from the Stiltjes property of the
continued fractions (see \cite{Andrews}, Lemma 5.5.2, pg. 256) we find 
$$ \frac{Q_{k-1}(t)}{P_k(t)} =  \(t + 1\cdot \(t+2\(t+\dots\(t+(k-1) t^{-1}\)^{-1}\dots
\)^{-1} \)^{-1} \)^{-1}. $$
Lemma~\ref{lem_pk} also implies that the following sequences are monotone: %% as $t>0$, $n=1, 2, \dots$:
$$ \frac{Q_{2n-1}(t)}{P_{2n}(t)} \,\,\text{is increasing and }\,\, \frac{Q_{2n}(t)}{P_{2n+1}(t)}
\,\,\text{is decreasing for } n=1, 2, \dots, \, t>0. $$

In the next section we will find explicit formulae for their coefficients.

\section{Coefficients of Laplace polynomials}\label{sec_PM}
Let us denote by $p_{k, m}, \,\, m=0, 1, \dots,$  the 
coefficients  of the polynomial $P_k(t)$  and by $q_{k,m}$ the coefficients of  $Q_k(t)$.
For the sake of convenience, we  introduce the polynomial $P_0(t)=1$. 

\begin{theorem}\label{th_main}
Denote $n=\frac{k-m}2$. The coefficients $p_{k,m}$ and $q_{k,m}$ satisfy the following relations.
If $m>k$ or $k-m\equiv 1 (\modd 2)$  then
\beq
p_{k,m}=q_{k, m}=0\quad  k,m =0, 1, \dots. \label{eq_mod1}
\eeq
If $k\equiv m (\modd 2)$ and $k\ge m$ then
\beq
p_{k,m} = \frac{k!}{m!\,\, 2^n\,\, n!}  \label{eq_form_pkm}
\eeq
and 
\beq
q_{k,m}=\frac{\(\frac{k+m}2\)!}{m!}\,\, 2^{-n} \sum_{j=0}^n {{k+1}\choose j}
\label{eq_form_qkm}
\eeq
\end{theorem}
\begin{proof}
At first, we establish Equations~(\ref{eq_mod1}) and (\ref{eq_form_pkm}).
We have
\begin{eqnarray*}
P_k(t) &=& \sum_{m=0}^k p_{k, m} t^m, \quad k=0, 1, \dots,  \\
Q_k(t) &=& \sum_{m=0}^k q_{k, m} t^m, \quad k=0, 1, \dots, 
\end{eqnarray*}
From (\ref{eq_P_k})  and (\ref{eq_Q_k}) we find that for $k=1, 2, \dots$ and
$m= 1, 2, \dots, k$ 
%% It is not difficult to see  from (\ref{eq_P_k}) that
\begin{eqnarray}
       p_{k+1, m} &=& p_{k, m-1} + (m+1)\cdot p_{k, m+1},   \label{eq_rec_pkm}  \\
q_{k, m} &=& p_{k, m} + (m+1)\cdot q_{k-1, m+1}.   \label{eq_rec_qkm}
\end{eqnarray}
If $m=0$ then Formula (\ref{eq_rec_pkm}) is understood as 
\beq p_{k+1, 0}  = p_{k, 1}, \quad k=2, 4, \dots. \label{eq_pk0} \eeq

\begin{table}[htp]
\begin{center}
\begin{tabular}{||c||c|c|c|c|c|c|c|c|c||}
\hline
  $k$  & \multicolumn{9}{|c|}{$m$}   \\ \cline{2-10}
       & $0$   & $1$  & $2$ & $3$ & $4$ & $5$ & $6$ & $7$ & $8$  \\ \hline \hline
 $0$  & $1$  & $0$ &  $0$  &  $\dots$   &    &   &   &      &  $0$     \\ \hline
 $1$  & $0$  & $1$ &  $0$  &  $\dots$   &    &   &   &      &  $0$     \\ \hline
 $2$  & $1$  & $0$ & $1$ &  $0$ & $\dots$ &     &     &     & $0$      \\ \hline
 $3$  & $0$  & $3$ & $0$ & $1$ & $0$ & $\dots$ &     &      &  $0$       \\ \hline
 $4$  & $3$  & $0$  & $6$ & $0$  & $1$ & $0$ & $\dots$  &  & $0$             \\ \hline
 $5$  & $0$  & $15$ & $0$ & $10$ & $0$ & $1$ & $0$ &  $\dots$   & $0$         \\ \hline
 $6$  & $15$ & $0$  & $45$ & $0$ & $15$ & $0$ & $1$ &  $0$   &  $0$        \\ \hline
 $7$  & $0$ & $105$  & $0$ & $105$ & $0$ & $21$ & $0$& $1$ &    $0$       \\ \hline
 $8$  & $105$ & $0$ & $420$ & $0$ & $210$ & $0$ & $28$ & $0$ & $1$  \\ \hline\hline
\end{tabular}
\end{center}
\begin{center}
\begin{minipage}{9cm }
\caption{Matrix $\fP=\Vert p_{k,m}\Vert$. }\label{tab_pkm} 
\end{minipage}\end{center}
\end{table}
% 
%\begin{remark} 
%It follows from Equation~(\ref{eq_form_pkm}) that
%$$ p_{k,0} = \frac{1}{\sqrt{2\pi}} \inmi x^k e^{-x^2/2}\d x. $$ 
% \end{remark}

%% It is natural to start our analysis with the elements of the matrix  $\fP$.
Probably, the most convenient way to find the general formula for the coefficient $p_{k, m}$
is to look at the diagonals $k-m=\const$ of the matrix $\fP$.
The following lemma proves the statements of Theorem~\ref{th_main} related to the coefficients
$p_{k,m}$.
\begin{lemma}\label{lem_pkm}
If $k-m\equiv 1(\modd 2)$, the coefficients $p_{k, m}=0.$
If $k-m=2n, n\in\Z$,
\beq
     p_{k, m} = 
\frac{k!}{ m!\cdot 2^{n}\,\, n! }
\label{eq_pkm} 
\eeq
\end{lemma}
\begin{proof} Note that from the relation
$p_{1, 0}=0$ and Equation~(\ref{eq_rec_pkm}) it follows that $p_{k, m}=0$ if $k-m$ is odd.
Consider the case $k-m$ is an even number.
The proof is carried out by double induction along the even diagonals of
the matrix $\fP$ (see Table~\ref{tab_pkm}).
%% using recurrent Equation~(\ref{eq_rec_pkm}). 
If $m=k$ then, obviously, $p_{k, m}=1$. 
Consider the diagonal  $k-m=2$. If $k=3$ then 
$p_{3, 1} =3$. Suppose 
$$ p_{k, k-2}={k\choose 2}\quad\text{for $k=3, 4, \dots, K-1$}.  $$
Then, using recurrent relation (\ref{eq_rec_pkm}), we obtain
$$ {{K-1}\choose 2} + (K-1)\cdot 1 = {K\choose 2}. $$  
Thus, we proved (\ref{eq_pkm}) for $k-m=2$. %% by inspection $p_{k, k-2}={k\choose 2}$. 

Let us now prove (\ref{eq_pkm}) in the general case. 
Suppose we already verified this relation
for $k-m=2n$, $n=0, 1, ..., N$. We shall prove  (\ref{eq_pkm}) for
$k-m=2(N +1)$. From (\ref{eq_pk0}) we obtain
$$ p_{2N+2, 0}=p_{2N+1, 1}=\frac{(2N +1)!}{2^{N}\,\, N !}. $$
 Taking into account the relation
$$ \frac{(2N +1)!}{2^{N}\,\, N !}= 
\frac{(2N +2)!}{2^{N+1}\,\, (N+1) !}, $$
we obtain that $p_{2N+2, 0}$ satisfies (\ref{eq_pkm});  
the induction base is verified.

Suppose for $l=0, 1, \dots, L-1$
$$ p_{2N +2+l, l}=  \frac{(2N +2+l)!}{l!\, 2^{N+1}\,\, (N+1) !}.$$
Then from (\ref{eq_rec_pkm}) we find
$$ p_{2N +2+L, L} = p_{2N+1+L, L-1}+ (L+1) p_{2N +1+L, L+1}. $$
The last element belongs to $2N$th diagonal. Therefore
$$ p_{2N +2+L, L} = \frac{(2N+1+L)!}{(L-1)!\cdot 2^{N+1} (N+1)! }
+ (L+1)\cdot \frac{(2N+1+L)!}{(L+1)!\cdot 2^{N} N! } .  $$
Finally, we derive
$$ p_{2N +2+L, L} = \frac{(2N+2+L)!}{L!\cdot 2^{N+1} (N+1)! } $$
as was to be proved.  \end{proof}

From (\ref{eq_pkm}) one can easily obtain
\begin{lemma}\label{lem_rec_pk}
The polynomials $P_k(t)$ satisfy the relation
\beq
P^\prime_k(t) = k P_{k-1}(t), \quad k=1, 2, \dots,  \label{eq_rec_pk}
\eeq
where $P_0(t)=1$.
\end{lemma}
\begin{proof} Equation~(\ref{eq_rec_pk}) follows from the relation
$$ p_{k, m} \cdot m = p_{k-1, m-1}\cdot k, \quad m=1, 2, \dots, k, $$
that follows directly from Lemma~\ref{lem_pkm}.  \end{proof}

%% Let us now find the coefficients $q_{k, m}$. 
\begin{table}[htp]
\begin{center}
\begin{tabular}{||c||c|c|c|c|c|c|c|c|c||}
\hline
  $k$  & \multicolumn{9}{|c|}{$m$}   \\ \cline{2-10}
       & $0$   & $1$  & $2$ & $3$ & $4$ & $5$ & $6$ & $7$ & $8$  \\ \hline \hline
 $0$  & $1$  & $0$ &  $0$  &  $\dots$   &    &   &   &      &  $0$     \\ \hline
 $1$  & $0$  & $1$ &  $0$  &  $\dots$   &    &   &   &      &  $0$     \\ \hline
 $2$  & $2$  & $0$ & $1$ &  $0$ & $\dots$ &     &     &     & $0$      \\ \hline
 $3$  & $0$  & $5$ & $0$ & $1$ & $0$ & $\dots$ &     &      &  $0$       \\ \hline
 $4$  & $8$  & $0$  & $9$ & $0$  & $1$ & $0$ & $\dots$  &  & $0$             \\ \hline
 $5$  & $0$  & $33$ & $0$ & $14$ & $0$ & $1$ & $0$ &  $\dots$   & $0$         \\ \hline
 $6$  & $48$ & $0$  & $87$ & $0$ & $20$ & $0$ & $1$ &  $0$   &  $0$        \\ \hline
 $7$  & $0$ & $279$  & $0$ & $185$ & $0$ & $27$ & $0$& $1$ &    $0$       \\ \hline
\end{tabular}
\end{center}
\begin{center}
\begin{minipage}{9cm }
\caption{Coefficients  $q_{k, m}$.}\label{tab_qkm}
\end{minipage}\end{center}
\end{table}
Let us now express the coefficients $q_{k, m}$ through the elements of the 
 matrix $\fP$.
%% Denote $2n=k-m$.
\begin{lemma}\label{prop_qkm}
The coefficients $q_{k, m}$ satisfy the relations
$$ q_{k,m}=0\quad \text{for} k-m\equiv 1(\modd 2), $$
\beq
m!\cdot q_{k, m}=\sum_{j=0}^n  (m+j)!\cdot p_{k-j, m+j}, \quad\text{for} k\equiv m\,(\modd 2).
\label{eq_qkm}
\eeq
\end{lemma}
\begin{proof} Let $\hat q_{k, m}= m! \cdot q_{k, m}$. Then from (\ref{eq_rec_qkm}) we obtain
\beq
    \hat q_{k, m}= m!\cdot p_{k, m} + \hat q_{k-1, m+1}.  \label{eq_q_p}
\eeq
Equation~(\ref{eq_qkm}) follows from (\ref{eq_q_p}).  \end{proof}
  
Now we are in the position to prove Formula (\ref{eq_form_qkm}).
From (\ref{eq_qkm}) we have
\begin{eqnarray*}
\hat q_{k,m} &=& \sum_{j=0}^n \frac{(k-j)!}{2^{n-j}\, (n-j)!} \\
&=& (k-n)!\cdot \sum_{j=0}^n {k-j\choose n-j} 2^{j-n} \\
&=& (k-n)!\cdot \sum_{i=0}^n {k-n+i\choose i} 2^{-i},
\end{eqnarray*}
where $2n=k-m$. Notice that $k-n=\frac{k+m}2$. Further simplification of the equation for 
$\hat q_{k,m}$ is based on
 the Cauchy integral representation for the binomial coefficients 
$$
{k\choose n}=\frac{1}{2\pi i} \oint_\gamma \frac{(1+z)^k}{z^{n+1}}\,\d z,
$$
where $\gamma$, the contour of integration, 
is a circle $\{ z: \abs{ z }=r_*\}$ of a sufficiently small radius, $r_*$ 
(say, $r_*=1/2$). Then we find
\begin{eqnarray*}
\sum_{i=0}^n {{k-n+i}\choose i} 2^{-i} &=&  \sum_{i=0}^n  2^{-i}\frac{1}{2\pi i} \oint_\gamma 
\frac{(1+z)^{k-n+i}}{z^{i+1} }\,\d z \\
&=&\frac{1}{2\pi i} \oint_\gamma \frac{(1+z)^{m+n}}{z}\cdot\sum_{j=0}^n
\(\frac{1+z}{2z}\)^j  \,\d z  \\
&=& 2\cdot \frac{1}{2\pi i} \oint_\gamma (1+z)^k\cdot \frac{1 - \( \frac{1+z}{2z} \)^{n+1} }{z-1} 
\d z   \\
&=& 2\cdot \frac{1}{2\pi i} \oint_\gamma \frac{(1+z)^k\cdot (2z)^{n+1}-(1+z)^{k+n+1} }
{(2z)^{n+1}(z-1)} \d z \\
&=& 2 \cdot \frac{1}{2\pi i} \oint_\gamma \frac{(1+z)^{m+n} }{z-1} \d z
+ 2^{-n} \cdot \frac{1}{2\pi i} \oint_\gamma \frac{(1+z)^{m+2n+1 } }{z^{n+1 }(1-z) }\d z.
\end{eqnarray*}
The first integral 
$$ \frac{1}{2\pi i} \oint_\gamma \frac{(1+z)^{m+n} }{z-1} \d z =0. $$
The second integral can be computed as follows. The integrand
\begin{eqnarray*}
 \frac{(1+z)^{m+2n+1} }{1-z } &=& (1+z)^{k+1}\sum_{j=0}^\iy z^j \\
 &=& \sum_{i=0}^{k+1}{{k+1}\choose i} z^i \cdot \sum_{j=0}^\iy z^j \\
&=& \sum_{l=0}^\iy \al_l z^l,
 \end{eqnarray*}
where $$ \al_l = \sum_{i=0}^{\min(k+1, l)} {{k+1}\choose i}.$$ 
Therefore
$$ \frac{1}{2\pi i} \oint_\gamma \frac{(1+z)^{m+2n+1 } }{z^{n+1 }(1-z) }\d z =
\al_n.$$
Since $n<k$, we have
$$ \al_n = \sum_{i=0}^n {k+1\choose i}.$$
Finally, we obtain
$$\hat q_{k,m}= (k-n)! \cdot 2^{-n}\sum_{i=0}^n {k+1\choose i}, $$
as was to be proved.  \end{proof}

\begin{corollary}\label{cor_q_k0}
\beq
q_{2n, 0} = 2^n n!\qquad n=0, 1, 2, \dots
\label{eq_q_k0}
\eeq
\end{corollary}
\begin{proof} In the case $m=0$, $k=2n$ we have 
$$ q_{k,m} = \hat q_{k, m}= n! \sum_{i=0}^n {2n+1\choose i}. $$
Equation~(\ref{eq_q_k0}) then follows from the identity
$$ \sum_{i=0}^n {2n+1\choose i}= 2^{2n}. $$
The corollary is thus proved.    \end{proof}

\section{Generating functions}\label{sec_gf}
In this section we compute the double generating functions of the 
Laplace polynomials.
%% matrices $\fP$ and $\fQ$. 
%% Introduce the polynomial $P_0(t)=1$ and denote
Denote
\beq 
\cP(s, t) = \sum_{k=0}^\iy\sum_{m=0}^\iy p_{k, m} \cdot t^m\frac{s^k}{k!}.  \label{eq_gf_p}
\eeq
\begin{lemma}
The series {\rm{(\ref{eq_gf_p})}} converges for all complex numbers $t$ and $s$ 
such that $\abs{t}<\iy$ and  $\abs{s}<\iy$. The generating function, $\cP(s, t)$, is
\beq
    \cP(s, t) =\exp\( st + \frac{s^2}{2}\).   \label{eq_gen_funct}
\eeq
\end{lemma}
\begin{proof}   Equation~(\ref{eq_gen_funct}) is known (see \cite{Kouba}). We shall prove 
(\ref{eq_gen_funct}) for the sake of completeness.
The series 
$$ G_m(s) = \sum_{k=0}^\iy p_{k, m} \frac{s^k}{k!} $$
converges for all $m\in\Z_+$. Denote $n=(k-m)/2$.
We have
$$\cP(s, t) = \sum_{m=0}^\iy  G_m(t) \cdot t^m $$
and the latter series converges for all $t$, $\abs{s}<\iy$.
Therefore,
\begin{eqnarray*}
 \cP(s, t) &=& \sum_{m=0}^\iy t^m \sum_{k=0}^\iy \frac{k!}{m!\, 2^n\, n!} \frac{s^k}{k!} \\
    &=& \sum_{m=0}^\iy t^m \sum_{n=0}^\iy \frac{s^{m+2n}}{m!\, 2^n\, n!}  \\
  &=& \sum_{m=0}^\iy \frac{(st)^m}{m!} \sum_{n=0}^\iy \frac{ s^{2n} }{2^n\, n!} \\
 %% &=&  \sum_{m=0}^\iy \frac{(st)^m}{m!} \exp\( \frac{s^2}{2}\) \\
  &=&  \exp\( st + \frac{s^2}{2}\). 
\end{eqnarray*}
Formula~(\ref{eq_gen_funct}) is thus proved.  \end{proof}

Let us now compute  the generating function of the Laplace polynomials $Q_k(t)$.
%% $\fQ=\Vert q_{k,m}\Vert$. 
Denote\footnote{Notice that the pair of polynomials $P_k$ and $Q_{k-1}$
determine the $k$th approximation of the Mill's ratio.} 
$$ \cQ(s, t)\bydef \sum_{k=0}^\iy \sum_{m=0}^\iy q_{k,m} t^m \frac{s^{k+1}}{(k+1)!}. $$
\begin{lemma}\label{lem_gf_q}
The generating function $\cQ(s, t)$ is
\beq
\cQ(s, t) = \sqrt{2\pi}\, e^{(s+t)^2/2}\cdot \(\Phi(s+t) - \Phi(t)\).
\label{eq_gfQ}
\eeq
\end{lemma}
\begin{proof}  We have, 
$$ \cQ(s, t) = \sum_{k=0}^\iy Q_{k}(t)\frac{s^{k+1}}{(k+1)!} \quad \text{and}\quad
    \cP(s, t) = \sum_{k=0}^\iy P_{k}(t) \frac{s^k}{k!}
$$
The Taylor series expansion for the function $R(t)$
 can be written as
$$ R(s+t) = \sum_{k=0}^\iy  \frac{\d^k R(t)}{\d t^k}\frac{s^k}{k!} .$$
The derivatives of the Mills ratio satisfy the equation
$$\frac{ \d^k R(t)}{\d t^k} = P_k(t) R(t) -  Q_{k-1}(t). $$
 Then we have
$$ R(s+t) = \sqrt{2\pi } e^{(s+t)^2/2} \bar\Phi(s+t) $$
and therefore
\begin{eqnarray*} 
    \sqrt{2\pi }\, e^{(s+t)^2/2}\cdot \bar\Phi(s+t) &=& \sum_{k=0}^\iy  \frac{\d^k R(t)}{\d t^k}\frac{s^k}{k!}  \\
       &=& R(t) \sum_{k=0}^\iy P_k(t) \frac{s^k}{k!} - 
    \sum_{k=0}^\iy Q_{k-1}(t)  \frac{s^k}{k!} \\
  &=& R(t) \cdot e^{st+s^2/2} - \cQ(s, t) \\
 &=& \sqrt{2\pi}\, e^{t^2/2}\cdot \bar\Phi(t) \cdot e^{st+s^2/2} - \cQ(s, t).
\end{eqnarray*}   
Therefore, %% from (\ref{eq_Lt_f}) we obtain
$$  \cQ(s, t) =\sqrt{2\pi}\, e^{(s+t)^2/2}\cdot\(\bar\Phi(t) -  \bar\Phi(s+t)\). $$
The latter relation implies Equation~(\ref{eq_gfQ}). \end{proof}

%% The following Ramanujan-style identity immediately follows from Lemma~\ref{lem_gf_q} after 
%% substitution of $r=s=1$ into 
Now we are in a position to derive a series of identities from the double generating function $\cQ(s, t)$.
\begin{corollary}\label{cor_Id_H}
\beq 
    \sum_{n=0}^\iy \sum_{m=0}^\iy \sum_{j=0}^{n}
    \frac{\(n+m\)!}{m! j! (m+2n+1-j)!}\,\, 2^{-n}  %% {{m+2n+1}\choose j}
    =\sqrt{2\pi} e^2 \(\Phi(2)-\Phi(1)\). \label{eq_Id_H}
\eeq
\end{corollary}
\begin{proof} Indeed,   substituting $s=r=1$ into (\ref{eq_gfQ}), we obtain
$$ \cQ(1,1) = \sqrt{2\pi}\, e^{2}\cdot \(\Phi(2) - \Phi(1)\). $$
On the other hand, by definition of the generating function, $\cQ(r, s)$
$$ \cQ(1, 1) = \sum_{k=0}^\iy \sum_{m=0}^\iy q_{k,m}  \frac{1}{(k+1)!}. $$
Using substitution, $k=m+2n$, from (\ref{eq_form_qkm}) and the latter equation we obtain
 the identity (\ref{eq_Id_H}). Corollary~\ref{cor_Id_H} is thus proved.  \end{proof}

Let us now establish the connection between the second identity, (\ref{eq_ident2})\footnote{That can be found in \cite{Riordan_CI}.}
and the generating function $\cQ(s, t)$.
\begin{corollary}\label{cor_ident2}
The coefficients $q_{2n, 0}$, of the Taylor series expansion for the function $\cQ(s, 0)$ are
$$   q_{2n, 0}=  \sum_{k=0}^n (-1)^k \frac{1}{2k+1} {n\choose k}.$$ \end{corollary}
\begin{remark} The latter formula for the coefficients $q_{2n, 0}$ implies  the  identity {\rm{(\ref{eq_ident2})}}
\end{remark}
\begin{proof} Let us substitute $t=0$ in (\ref{eq_gfQ}). Then we obtain
\begin{eqnarray*}
 \cQ(s, 0) &=& \sqrt{2\pi}\exp\( s^2/2\)\cdot\( \Phi(s)-\frac12 \) \\
          &=& \sum_{k=0}^\iy Q_k(0)\frac{s^{k+1}}{(k+1)!} \\
        &=& \sum_{n=0}^\iy q_{2n, 0}\frac{s^{2n+1}}{(2n+1)!}. 
\end{eqnarray*}
From (\ref{eq_q_k0}) we find
$$ \sum_{n=0}^\iy q_{2n, 0}\frac{s^{2n+1}}{(2n+1)!} = 
   \sum_{n=0}^\iy s^{2n+1}\frac{2^n\cdot n!}{(2n+1)!}. $$ 
Therefore
\beq 
   \cQ(s, 0) = \sum_{n=0}^\iy s^{2n+1}\frac{2^n\cdot n!}{(2n+1)!}.  \label{eq_cor_ident2}
\eeq
The standard normal cdf satisfies the relation (see \cite{Abram})
$$ \Phi(s) = \frac12 + \frac{1}{\sqrt{2\pi}}\sum_{n=0}^\iy \frac{(-1)^n s^{2n+1}}{2^n n! (2n+1)}. $$
Then we derive
\begin{eqnarray*}
\cQ(s, 0) &=& \sqrt{2\pi}\exp\( s^2/2\)\cdot\( \Phi(s)-\frac12 \)  \\ 
    &=& \sum_{j=0}^\iy \frac{s^{2j}}{2^j j!} \cdot \sum_{n=0}^\iy \frac{(-1)^n s^{2n+1}}{2^n n! (2n+1)} \\
  &=& \sum_{n=0}^\iy s^{2n+1}  \sum_{j=0}^n  \frac{(-1)^j }{2^j j!\, 2^{n-j} (n-j)!\, (2j+1)} \\
  &=& \sum_{n=0}^\iy s^{2n+1} \frac{1}{2^n\, n!} \sum_{j=0}^n  {n\choose j}\frac{(-1)^j }{(2j+1)} .
\end{eqnarray*}
Comparing the latter equation with (\ref{eq_cor_ident2}) we derive the identity (\ref{eq_ident2}).  
\end{proof}
 
\section{Laplace and Hermite polynomials}\label{sec_LH}
Neither polynomials $P_k(t)$ nor $Q_k(t)$  form a system of orthogonal polynomials in the real domain.
Nevertheless, the polynomials $P_k(t)$ are intimately connected to the Hermite polynomials of the complex argument.
% 
%\rmk The Laplace polynomials, $P_k(t)$, are closely related to the Hermite polynomials, 
%$H_k(t)$, whose generating function
%$$ \cH(t, s)\bydef \sum_{k=0}^\iy H_k(t) \frac{s^k}{k!}=e^{st-s^2/2}. $$
Indeed, the Hermite and the Laplace polynomials can be defined as
$$ 
   H_k(t) \bydef (-1)^k e^{t^2/2} \,\frac{\d^k \exp(-t^2/2)}{\d t^k}; \quad 
   P_k(t) \bydef e^{-t^2/2} \,\frac{\d^k \exp(t^2/2)}{\d t^k}.
$$
 Then we obtain
%% \load
\beq 
    P_k(t) = (-\mathit i)^k H_k( {\mathit i} t).   \label{eq_P_H_k}
\eeq
This relation (see also \cite{Kouba}) allows us to reformulate the
classical  results obtained for the Hermite polynomials in terms of
the Laplace polynomials. In particular, one can easily derive the generating function, $\cP(s, t)$,
from the generating function
$$ \cH(t, s)\bydef \sum_{k=0}^\iy H_k(t) \frac{s^k}{k!}=e^{st-s^2/2} $$
for the Hermite polynomials.

\noindent Another useful fact about polynomials $P_k(t)$ is formulated in the following 
\begin{proposition}\label{cor_deq} The polynomial $P_k(t)$ satisfies the differential equation
\beq y^{\prime\prime} +t y^\prime -k y=0. \label{eq_deq_pol_Pk} \eeq
\end{proposition}
\begin{proof} This result can be derived from the corresponding differential equation for the Hermite polynomials.
One can also derive Equation~(\ref{eq_deq_pol_Pk}) from (\ref{eq_rec_pkm}) and Lemma~\ref{lem_rec_pk}.
\end{proof}

Polynomials $P_k(t)$ form basis in the space of polynomials on the real line. In particular, the polynomial
$Q_k(t)$ can be represented as a linear combination
\beq
    Q_k(t)=\sum_{j\in J(k)} \be_{k, j} P_j(t),  \label{eq_Q_P_k}
\eeq 
where the set of indices $J(k)=\{ j: 0\le j\le k,\, k\equiv\,j \(\modd 2\) \}$.
\begin{proposition}\label{prop_bet_kj}
For every $k=0, 1, 2, \dots $, there is a unique representation (\ref{eq_Q_P_k}) for the
Laplace polynomials $Q_k(t)$; the coefficients $\be_{k, j}$ are
\beq
  \be_{k, j} = \frac{n !}{ j! }, \qquad\text{where $n=(k+j)/2, \,\,j\in J(k)$.}
\label{eq_be_kj}
\eeq
\end{proposition}
\proof Equation~(\ref{eq_be_kj}) can be proved by induction. For $k=1$ $\be_{1, 1}=1$.
Suppose the Proposition is proved for $k\le k_\ast$. 
\begin{table}[htp]
\begin{center}
\begin{tabular}{||c||c|c|c|c|c|c|c|c|c||}
\hline
  $k$  & \multicolumn{9}{|c|}{$j$}   \\ \cline{2-10}
       & $0$   & $1$  & $2$ & $3$ & $4$ & $5$ & $6$ & $7$ & $8$  \\ \hline \hline
 $0$  & $1$  & $0$ &  $0$  &  $\dots$   &    &   &   &      &  $0$     \\ \hline
 $1$  & $0$  & $1$ &  $0$  &  $\dots$   &    &   &   &      &  $0$     \\ \hline
 $2$  & $1$  & $0$ & $1$ &  $0$ & $\dots$ &     &     &     & $0$      \\ \hline
 $3$  & $0$  & $2$ & $0$ & $1$ & $0$ & $\dots$ &     &      &  $0$       \\ \hline
 $4$  & $2$  & $0$  & $3$ & $0$  & $1$ & $0$ & $\dots$  &  & $0$             \\ \hline
 $5$  & $0$  & $6$ & $0$ & $4$ & $0$ & $1$ & $0$ &  $\dots$   & $0$         \\ \hline
 $6$  & $6$ & $0$  & $12$ & $0$ & $5$ & $0$ & $1$ &  $0$   &  $0$        \\ \hline
 $7$  & $0$ & $24$  & $0$ & $20$ & $0$ & $6$ & $0$& $1$ &    $0$       \\ \hline
\end{tabular}
\end{center}
\begin{center}
\begin{minipage}{9cm }
\caption{Coefficients  $\be_{k, j}$.}\label{tab_be_kj}
\end{minipage}\end{center}
\end{table}
Let us prove that $Q_{k_\ast+1}$  satisfies Equation~(\ref{eq_Q_P_k}) with $\be_{k_\ast+1, j}$ defined by
(\ref{eq_be_kj}). 
\begin{proof} We have from Lemma~\ref{lem_pk}
\begin{eqnarray*} 
Q_{k_\ast+1}(t) &=& t Q_{k_\ast}(t) + (k_\ast+1) Q_{k_\ast-1}(t) \\
 &=& t \sum_{j\in J(k_\ast)} \be_{k_\ast,j} P_j(t) + (k_\ast +1) \sum_{j\in J(k_\ast-1)} \be_{k_\ast-1,j} P_j(t) \\
&=& \sum_{j\in J(k_\ast)} \be_{k_\ast,j} \(P_{j+1}(t) - j P_{j-1}(t)\) + (k_\ast + 1)
\sum_{j\in J(k_\ast-1)} \be_{k_\ast-1,j} P_j(t).
\end{eqnarray*}
Thus the induction step will be proved if we show that the coefficients $\be_{k,j}$ satisfy the relation
\beq
\be_{k+1,j}=\be_{k,j-1} - (j+1)\be_{k, j+1} + (k+1)\be_{k-1, j}.  \label{eq_bet_rec}
\eeq
But it is easy to verify that the coefficients $\be_{k,j}$ defined by (\ref{eq_be_kj}) satisfy (\ref{eq_bet_rec}).
\end{proof}

\section{Laplace polynomials and matching numbers}\label{sec_LPmn}
Consider a complete graph $G_k$, i. e. the graph with $k$ vertices such that every two vertices 
are connected by a single edge. Recall (see \cite{Andrews}) that  a set of edges sharing no common vertices is called
matching.
 Let $M(k, m)$ be a number of matchings with $m$ edges  in the complete graph, $G_k$.
The following relation between  $M(k, m)$ and the coefficients of the Laplace polynomials, $p_{k, j}$,
follows from the  very well known formula for the matching number (see \cite{Andrews}, \cite{Riordan_CI}): 
\begin{lemma}\label{lem_ap1}
\beq
M(k, m) = p_{k, k-2m}.  \label{eq_ap1}
\eeq
\end{lemma}
\begin{proof} We have
$$
M(k, m)={k\choose{2m}} (2m-1)\cdot (2m-3)\cdot \dots \cdot 3 \cdot 1 = 
\frac{ k!}{(k-2m)! \cdot 2^m \cdot m!}.
$$
Then from (\ref{eq_form_pkm}) we obtain (\ref{eq_ap1}).  
The lemma is thus proved. \end{proof}

\noindent Denote $k_\ast = [k/2]$  the number of edges in the maximal matching. Let
$$\cM_k(t) \bydef \sum_{m=0}^{k_\ast} M(k, m) t^m, $$
be
the generating function of the matching numbers of the graph $G_k$.
Then from Lemma~\ref{lem_ap1} we obtain
\begin{corollary} For $k\ge 2$
$$  \cM_k(t) = t^{k/2} P_k\( t^{-1/2}\). $$
\end{corollary}

%% \bibliographystyle{elsarticle-harv}
%% \bibliography{TailDistrib}
 
\end{document}